\documentclass[11pt,a4paper]{article}
\usepackage{epsf,epsfig,amssymb,amsfonts,amsgen,amsmath,amstext,amsbsy,amsopn,amsthm,lineno}
\usepackage{color,comment}
\usepackage{graphicx,tikz}
\usepackage{hyperref}
\usepackage{multirow}
\usepackage{cite}

\newtheorem{theorem}{Theorem}[section]

\newtheorem{prop}[theorem]{Proposition}
\newtheorem{lemma}[theorem]{Lemma}

\theoremstyle{definition}
\newtheorem{definition}[theorem]{Definition}
\newtheorem{remark}[theorem]{Remark}

\newtheorem{conjecture}[theorem]{Conjecture}

\newcommand\Cay{\mathrm{Cay}}
\newcommand\Aut{\mathrm{Aut}}
\newcommand\Inn{\mathrm{Inn}}
\newcommand\Out{\mathrm{Out}}
\newcommand\Soc{\mathrm{Soc}}

\newcommand\SL{\mathrm{SL}}
\newcommand\PSL{\mathrm{PSL}}
\newcommand\PGL{\mathrm{PGL}}
\newcommand\PSU{\mathrm{PSU}}
\newcommand\PSp{\mathrm{PSp}}
\newcommand\Sp{\mathrm{Sp}}
\newcommand\PSO{\mathrm{PSO}}

\newcommand\Om{\mathrm{\Omega}}

\newcommand\G{\mathrm{G}}

\newcommand\Cy{\mathrm{C}}
\newcommand\diag{\mathrm{diag}}
\newcommand\I{I}
\newcommand\T{\mathrm{T}}
\newcommand\N{\mathbf{N}}
\newcommand\C{\mathbf{C}}
\newcommand\Z{\mathbf{Z}}
\newcommand\F{\mathbb{F}}

\newcommand\ppd{\mathrm{ppd}}
\newcommand\GL{\mathrm{GL}}
\newcommand\GU{\mathrm{GU}}

\newcommand\rmO{\mathrm{O}}
\newcommand\GAL{\mathrm{\Gamma L}}
\newcommand\PGAL{\mathrm{P\Gamma L}}
\newcommand\PR{\mathrm{P}}
\newcommand\Inv{\mathrm{Inv}}

\usepackage{xcolor}
\usepackage[normalem]{ulem}
\newcommand{\cha}[2]{\textcolor{blue}{\sout{#1}}\textcolor{red}{#2}}

\long\def\delete#1{}

\tikzstyle{vertex}=[circle, draw, inner sep=0pt, minimum size=3pt]

\usepackage{mathdots}

\numberwithin{equation}{section}
\allowdisplaybreaks

\title{\textbf{Cubic graphical regular representations of some classical simple groups}}
\author{Binzhou Xia\footnote{School of Mathematics and Statistics, The University of Melbourne, Parkville, VIC 3010, Australia (\texttt{binzhoux@unimelb.edu.au, zhesz@student.unimelb.edu.au, sanming@unimelb.edu.au})}, ~Shasha Zheng\footnotemark[1], ~Sanming Zhou\footnotemark[1]}

\date{}

\begin{document}

\maketitle
\openup 0.5\jot

\begin{abstract}
A graphical regular representation (GRR) of a group $G$ is a Cayley graph of $G$ whose full automorphism group is equal to the right regular permutation representation of $G$. In this paper we study cubic GRRs of $\PSL_{n}(q)$ ($n=4, 6, 8$), $\PSp_{n}(q)$ ($n=6, 8$), $\PR\Omega_{n}^{+}(q)$ ($n=8, 10, 12$) and $\PR\Omega_{n}^{-}(q)$ ($n=8, 10, 12$), where $q = 2^f$ with $f \ge 1$.
We prove that for each of these groups, with probability tending to $1$ as $q \rightarrow \infty$, any element $x$ of odd prime order dividing $2^{ef}-1$ but not $2^{i}-1$ for each $1 \le i < ef$ together with a random involution $y$ gives rise to a cubic GRR, where $e=n-2$ for $\PR\Omega_{n}^{+}(q)$ and $e=n$ for other groups. Moreover, for sufficiently large $q$, there are elements $x$ satisfying these conditions, and for each of them there exists an involution $y$ such that $\{x,x^{-1},y\}$ produces a cubic GRR. This result together with certain known results in the literature implies that except for $\PSL_2(q)$, $\PSL_3(q)$, $\PSU_3(q)$ and a finite number of other cases, every finite non-abelian simple group contains an element $x$ and an involution $y$ such that $\{x,x^{-1},y\}$ produces a GRR, showing that a modified version of a conjecture by Spiga is true. Our results and several known results together also confirm a conjecture by Fang and Xia which asserts that except for a finite number of cases every finite non-abelian simple group has a cubic GRR.

\medskip

\smallskip

\end{abstract}

\section{Introduction}

Let $G$ be a group whose identity element is denoted as $1$, and let $S$ be a subset of $G$ such that $1\notin S$ and $S^{-1}=S$, where $S^{-1}=\{x^{-1}: x \in S\}$. The \emph{Cayley graph} of $G$ with \emph{connection set} $S$, denoted by $\Cay(G,S)$, is defined as the graph with vertex set $G$ such that $x$ and $y$ are adjacent if and only if $yx^{-1}\in S$. A graph $\Gamma$ is said to be a \emph{graphical regular representation} (\emph{GRR}) of a group $G$ if its full automorphism group $\Aut(\Gamma)$ is isomorphic to $G$ and acts on the vertex set of $\Gamma$ as a regular permutation group \cite{Godsil1981}. It is well known that, for a Cayley graph $\Cay(G,S)$, the right regular permutation representation of $G$ is a subgroup of $\Aut(\Cay(G,S))$. Thus $\Cay(G,S)$ is a GRR of $G$ if and only if $\Aut(\Cay(G,S))\cong G$. Since a GRR of a group $G$ must be isomorphic to a Cayley graph of $G$, any GRR of $G$ is isomorphic to a Cayley graph of $G$ whose full automorphism group is equal to the right regular permutation representation of $G$. If there exists such a Cayley graph, then $G$ is said to admit a GRR.

It is natural to ask which finite groups admit GRRs. This question was studied in a series of papers (see,  for example, \cite{Sabidussi1964, Chao1964, Nowitz1968, Watkins1971}), and eventually a complete characterization was obtained by Godsil in \cite{Godsil1981}: Apart from abelian groups of exponent at least three, generalized dicyclic groups and thirteen other groups, every finite group admits a GRR. There is also special interest in studying which finite groups admit GRRs of a prescribed valency (see, for example, \cite{Godsil1983, Fang2002, Xu2004, Spiga2018, Xia2016, Xia2020, Xia20202}). In the case of valency three, Fang, Li, Wang and Xu \cite{Fang2002} conjectured that every finite non-abelian simple group admits a cubic GRR. However, in \cite{Xia2016}, Fang and Xia found that $\PSL_2(7)$ is a counterexample to this conjecture. Meanwhile, they proposed the following conjecture in the same paper.

\begin{conjecture}(\cite[Conjecture 4.3]{Xia2016})\label{C1}
Except a finite number of cases, every finite non-abelian simple group has a cubic GRR.
\end{conjecture}

It is known that if $\Cay(G,S)$ is  a GRR of $G$, then we necessarily have $\Aut(G,S)=1$ and $G=\langle S\rangle$, where
$$\Aut(G,S)=\{\alpha\in \Aut(G): S^{\alpha}=S\}.$$
With the help of \cite[Theorem 1.3]{Fang2002} and \cite[Theorem 1.1]{Spiga2018}, a list of finite non-abelian simple groups $G$ for which a cubic Cayley graph $\Cay(G,S)$ is a GRR if and only if $\langle S\rangle=G$ and $\Aut(G,S)=1$ was given in \cite{Xia2020}. Among other groups this list contains all sporadic simple groups, all simple groups of exceptional Lie type, some alternating groups, and some classical groups. On the other hand, in \cite[Theorem 1.1]{Leemans2017}, Leemans and Liebeck proved that for every finite non-abelian simple group $G$ except $A_7$, $\PSL_2(q)$, $\PSL_3(q)$ and $\PSU_3(q)$, there exists a pair of generators $(x,y)$ of $G$ where $y$ is an involution such that $\Aut(G,\{x,x^{-1},y\})=1$. Combining this with the above-mentioned list in \cite{Xia2020}, we see that most families of finite non-abelian simple groups have cubic GRRs. In particular, all sporadic simple groups and simple groups of exceptional Lie type admit cubic GRRs. Alternating groups of degree at least $19$ \cite{Godsil1983} and some classical groups \cite{Xia2016, Xia2020, Xia20202} are also known to admit cubic GRRs. Thus, to settle Conjecture \ref{C1}, the remaining families of groups that need to be considered are: $\PSU_3(q)$ with $q\neq2$, $\PSL_4(q)$, $\PSL_6(q)$ where $\gcd(6,q-1)=1$, $\PSL_8(q)$, $\PSp_6(q)$, $\PSp_8(q)$, $\PR\Omega_{8}^{\pm}(q)$, $\PR\Omega_{10}^{\pm}(q)$, and $\PR\Omega_{12}^{\pm}(q)$ with $q$ even. In this paper, we prove the existence of cubic GRRs of these remaining families of groups except $\PSU_3(q)$. Cubic GRRs of $\PSU_3(q)$ with $q \geq 4$ are known to exist as proved in a recent paper \cite{Li2021}.

For a pair of positive integers $(a,m)$ with $a\geq 2$ and $m\geq 2$, we call a prime divisor $r$ of $a^m-1$ a \emph{primitive prime divisor} of $(a,m)$ if $r$ does not divide $a^i-1$ for every positive integer $i<m$.  Denote the set of primitive prime divisors of $(a,m)$ by $\ppd (a,m)$.
By Zsigmondy's theorem (see,  for example, \cite[Chapter IX, Theorem 8.3]{Blackburn1982}), if $(a,m)\neq (2^k-1,2),(2,6)$, then $\ppd (a,m)$ is not empty.

The main result in this paper is as follows, where a random involution is meant to be an involution chosen uniformly at random from all involutions of the group under consideration.

\begin{theorem}\label{T0}
Let $G$ be a finite simple group and $r\in \ppd (2,ef)$, where $G$ and $e$ are given in Table \ref{Ge} with $q=2^f$, $f \ge 1$, and let $x$ be any fixed element of $G$ with order $r$. Then for a random involution $y$ of $G$ the probability of $\Cay(G,\{x,x^{-1},y\})$ being a GRR of $G$ approaches $1$ as $q$ approaches infinity.
\end{theorem}

\begin{table}[h]
\centering
		\begin{tabular}{l|l|l}
			\hline
			$G$ & Conditions & $e$\\ \hline
            $\PSL_{n}(q)$ & $n=4, 6\text{ or }8$, $q$ even &  $n$\\
			$\PSp_{n}(q)$ & $n=6\text{ or }8$, $q$ even &  $n$ \\
            $\PR\Omega_{n}^{+}(q)$ & $n=8, 10\text{ or }12$, $q$ even &  $n-2$\\
            $\PR\Omega_{n}^{-}(q)$ & $n=8, 10\text{ or }12$, $q$ even &  $n$\\\hline
		\end{tabular}
\caption{The pairs $(G, e)$ in Theorem \ref{T0}}
\label{Ge}
\end{table}

Spiga conjectured \cite[Conjecture 1.4]{Spiga2018} that for finite non-abelian simple groups, the proportion of cubic Cayley graphs of a given group that are GRRs approaches $1$ as the order of the group approaches infinity. Theorem \ref{T0} provides evidence in some sense (without considering connection sets of size three other than those considered in this theorem) to support this conjecture for the families of groups in Table \ref{Ge}.

\begin{remark}
Since $\ppd (2,6)=\emptyset$, for $G=\PSL_{6}(2), \PSp_{6}(2), \PR\Omega_{8}^{+}(2)$, we have $ef = 6$ and hence there is no element of $G$ with order in $\ppd (2,ef)$. However, if $G$ is any other group in Table \ref{Ge}, then such an element of $G$ with order $r$ exists by Zsigmondy's theorem and Cauchy's theorem. Moreover, by \cite[Section 8]{Praeger1999} the proportion of elements of $G$ whose orders are multiples of some primitive prime divisor of $(2,ef)$ or $(q,e)$ is no less than $1/(e+1)$.
\end{remark}

We will also prove the following result along the way in the proof of Theorem \ref{T0}.

\begin{theorem}\label{T}
Let $G$ be a finite simple group and $r\in \ppd (2,ef)$, where $G$ and $e$ are given in Table \ref{Ge} with $q=2^f$ for some positive integer $f$. Suppose that $q$ satisfies the condition in Table \ref{cq}. Then $G$ contains elements of order $r$, and moreover for any fixed element $x$ of $G$ with order $r$ there exists an involution $y$ in $G$ such that $\Cay(G,\{x,x^{-1},y\})$ is a GRR of $G$.
\end{theorem}

\begin{table}[h]
\centering
		\begin{tabular}{c|c|c|c|c|c|c}
			\hline
			$G$ & $\PSL_{4}(q)$ & $\PSL_{6}(q)$ & $\PSL_{8}(q)$ & $\PSp_{6}(q)$ & $\PSp_{8}(q)$ &  \\ \hline
            Condition & $q\geq 2^{3}$ & $q\geq 2^{7}$ & $q\geq 2^{2}$ & $q\geq 2^{6}$ & $q\geq 2^{4}$ &    \\ \hline
            \multicolumn{7}{l}{}\\ \hline
           $G$  & $\PR\Omega_{8}^{+}(q)$ & $\PR\Omega_{10}^{+}(q)$ & $\PR\Omega_{12}^{+}(q)$ & $\PR\Omega_{8}^{-}(q)$ & $\PR\Omega_{10}^{-}(q)$ & $\PR\Omega_{12}^{-}(q)$  \\ \hline
            Condition & $q\geq 2^{4}$ & $q\geq 2^{5}$ & $q\geq 2^{2}$ & $q\geq 2^{3}$ & $q\geq 2^{2}$ & $q\geq 2^{2}$  \\ \hline
		\end{tabular}
\caption{Conditions for $q$ in Theorem \ref{T}}
\label{cq}
\end{table}

In \cite[Conjecture 1.3]{Spiga2018}, Spiga conjectured that except $\PSL_2(q)$ and a finite number of other cases, every finite non-abelian simple group $G$ contains an element $x$ and an involution $y$ such that $\Cay(G,\{x,x^{-1},y\})$ is a GRR of $G$. It turns out that this conjecture is not true as both $\PSL_3(q)$ and $\PSU_3(q)$ form infinite families of counterexamples. In fact, by \cite[Theorem 4, Corollary 6]{Breda2021}, for $G=\PSL_3(q)$ or $\PSU_3(q)$ and any pair of generators $(x,y)$ of $G$ where $y$ is an involution, $\Aut(G,\{x,x^{-1},y\})$ is always nontrivial and therefore the connection set of any cubic GRR of $G$ (if it exists) consists of three involutions. Theorem \ref{T} together with several known results implies that we can save Spiga's conjecture (\cite[Conjecture 1.3]{Spiga2018}) by adding $\PSL_3(q)$ and $\PSU_3(q)$ to the list of exceptional groups. More specifically, by Theorem \ref{T}, \cite[Theorem 1.3]{Fang2002}, \cite[Theorem 1.1]{Spiga2018}, \cite[Theorem 1.1]{Leemans2017}, and some known results on alternating groups \cite{Godsil1983} and classical groups \cite{Xia20202}, we obtain the following result.

\begin{theorem}\label{T2}
Except $\PSL_2(q)$, $\PSL_3(q)$, $\PSU_3(q)$ and a finite number of other cases, every finite non-abelian simple group $G$ contains an element $x$ and an involution $y$ such that $\Cay(G,\{x,x^{-1},y\})$ is a GRR of $G$.
\end{theorem}

Note that, by \cite[Theorem 1.3]{Xia2016}, \cite[Theorem 1.2]{Xia2020}, \cite[Theorem 1.2]{Li2021} and the proofs of them, $\PSL_2(q)$ with $q\notin\{2,3,7\}$, $\PSL_3(q)$ with $q\neq2$ and $\PSU_3(q)$ with $q\geq4$ admit cubic GRRs with connection sets consisting of three involutions. Combining this with Theorem \ref{T2}, we confirm Conjecture \ref{C1} as follows.

\begin{theorem}\label{T3}
Let $G$ be a finite non-abelian simple group. Except a finite number of cases, $G$ has a cubic GRR.
\end{theorem}

The rest of this paper is devoted to the proof of Theorems \ref{T0} and \ref{T}. We will begin with two preliminary results and a hint to our approaches in the next section. This will be followed by three sections which deal with linear groups, symplectic groups and orthogonal groups, respectively. The main body of the paper consists of the proof of Propositions \ref{PL4}, \ref{PL68}, \ref{PS68}, \ref{PO+} and \ref{PO-} from which Theorems \ref{T0} and \ref{T} follow immediately.

\section{Preliminaries and methodologies}

For a subset $A$ of a group $G$, denote by $\I_2(A)$ the set of involutions of $G$ in $A$, and set $$i_2(A)=|\I_2(A)|.$$ In particular, $i_2(G)$ is the number of involutions of $G$.

Let $G$ be a group as in Table \ref{Ge}, and let $x$ be an element of $G$ with order $r\in \ppd (2,ef)$. Define
\begin{align*}
  &K(x)=\{y \in \I_2(G): G=\langle x,y\rangle\},\\
  &L(x)=\{y \in \I_2(G): \Aut(G,\{x,x^{-1},y\})=1\},\\
  &\Inv(x)=\{\alpha \in \I_2(\Aut(G)): x^{\alpha}=x^{-1}\},
\end{align*}
and let $\mathcal{M}(x)$ be the set of maximal subgroups of $G$ containing $x$. Denote by $I_{m}$ the $m\times m$ identity matrix.

\begin{definition}
Let $n\geq3$ and $q=2^f$ with $f$ a positive integer. Given an integer $\ell$ with $1\leq \ell\leq n/2$, define the involution $j_{\ell}(n)$ of $\SL_{n}(q)$ by
\[
j_{\ell}(n)=\left(
  \begin{array}{ccc}
    I_{\ell} & O & O\\
    O & I_{n-2\ell} & O\\
    I_{\ell} & O & I_{\ell}
  \end{array}
\right)\cha{,}{}
\]
and call it the \emph{Suzuki form} of its conjugacy class in $\SL_{n}(q)$.
\end{definition}

It is known that these involutions form a complete set of representatives for the conjugacy classes of involutions in the corresponding group. See \cite[Section 4]{Aschbacher1976} for details.

\begin{table}[t]
\centering
		\begin{tabular}{l|l|l|l|l}
			\hline
			$G$ & Conditions & $|\N_{G}(\langle x \rangle)|$  & $i(G)$ & $u(G)$\\ \hline
            $\PSL_{4}(q)$ & $q\geq 4$ even & $4(q^4-1)/(q-1)$  & $q^5(q^3-1)$ &  $22q^3$\\
            $\PSL_{6}(q)$ & $q\geq 4$ even & $6(q^6-1)/((q-1)\gcd(6,q-1))$  & $q^{18}/2$ &  $32q^{5}$\\
            $\PSL_{8}(q)$ & $q$ even & $8(q^8-1)/(q-1)$  & $q^{32}/2$ &  $64q^7$\\
            $\PSp_{6}(q)$ & $q\geq 4$ even & $6(q^{3}+1)$  & $q^{12}$ &  $25q^{3}$\\
            $\PSp_{8}(q)$ & $q$ even & $8(q^{4}+1)$  & $q^{20}$ &  $34q^{4}$\\
            $\PR\Omega_{8}^{+}(q)$ & $q\geq 4$ even & $6(q^{3}+1)(q+1)$  & $q^{16}/2$ &  $61q^{4}$\\
            $\PR\Omega_{10}^{+}(q)$ & $q$ even & $8(q^{4}+1)(q+1)$  & $q^{24}/2$ &  $102q^{5}$\\
            $\PR\Omega_{12}^{+}(q)$ & $q$ even & $10(q^{5}+1)(q+1)$  & $q^{36}/2$ &  $124q^{6}$\\
            $\PR\Omega_{8}^{-}(q)$ & $q$ even & $4(q^{4}+1)$  & $q^{16}/2$ &  $34q^{4}$\\
            $\PR\Omega_{10}^{-}(q)$ & $q$ even & $5(q^{5}+1)$  & $q^{24}/2$ &  $42q^{5}$\\
            $\PR\Omega_{12}^{-}(q)$ & $q$ even & $6(q^{6}+1)$  & $q^{36}/2$ &  $49q^{6}$\\\hline
		\end{tabular}
\caption{The tuples $(G, |\N_{G}(\langle x \rangle)|, i(G), u(G))$}
\label{IU}
\end{table}

\begin{lemma}
\label{lem:24}
Let $G$, $q$ and $x$ be as in Theorem \ref{T0}, and let $i(G)$ and $u(G)$ be as in Table \ref{IU}. Then
\begin{align}
i_2(G) & \geq  i(G), \label{eq:iG} \\
|\Inv(x)| & \leq  u(G). \label{eq:uG}
\end{align}
\end{lemma}

\begin{proof}
Note that $i_2(\PSL_{n}(q)) =i_2(\GL_{n}(q))$ when $q$ is even. From the proof of \cite[Proposition 4.1]{Liebeck1996}, we know that
\begin{equation}
\label{cis}
|\C_{\GL_{n}(q)}(j_{\ell}(n))|=q^{\ell(2n-3\ell)}|\GL_{\ell}(q)||\GL_{n-2\ell}(q)|
\end{equation}
for $j_{\ell}(n)\in\GL_{n}(q)$, where $1\leq \ell\leq n/2$ and $|\GL_{0}(q)|=1$ by convention. Thus
\[
i_2(\PSL_{4}(q)) = \frac{|\GL_{4}(q)|}{q^5|\GL_{1}(q)||\GL_{2}(q)|}+\frac{|\GL_{4}(q)|}{q^4|\GL_{2}(q)|}>q^5(q^3-1),
\]
and
\[
i_2(\PSL_{n}(q)) > \frac{|\GL_{n}(q)|}{q^{n^2/4}|\GL_{n/2}(q)|} > \frac{q^{\frac{n^2}{2}}}{2}
\]
for $n=6, 8$.

When $q$ is even, by \cite[Theorem 4.3]{Fulman2017} we have
\[
i_2(\PSp_{n}(q)) > q^{\frac{n^2}{4}+\frac{n}{2}}
\]
for $n=6, 8$. (Note that in \cite[Theorem 4.3]{Fulman2017} the identity element is viewed as an involution.)
When $q$ is even, by \cite[Table 3.5.1]{Burness2016} we have
\[
i_2(\PR\Omega_{10}^{\pm}(q))>\frac{|\PR\Omega_{10}^{\pm}(q)|}{|\C_{\PR\Omega_{10}^{\pm}(q)}(c_{4})|}> \frac{q^{24}}{2}
\]
and
\[
i_2(\PR\Omega_{n}^{\pm}(q))>\frac{|\PR\Omega_{n}^{\pm}(q)|}{|\C_{\PR\Omega_{n}^{\pm}(q)}(c_{n/2})|}> \frac{q^{\frac{n^2}{4}}}{2}
\]
for $n = 8, 12$, where the involutions $c_{s}$ are as defined in \cite[Section 3.5.4]{Burness2016}.

Similar to the proof of \cite[Theorem 1.3]{Xia20202}, since every involution in $\N_{\Aut(G)}(\langle x \rangle)$ projects to an involution or the identity in $\Aut(G)/(\Aut(G)\cap \PGL(V))$ while $i_{2}(\Aut(G)/(\Aut(G)\cap \PGL(V)))\leq 3$, we have
\begin{align*}
  |\Inv(x)| &\leq i_2(\N_{\Aut(G)}(\langle x \rangle)) \\
  &\leq (i_{2}(\Aut(G)/(\Aut(G)\cap \PGL(V)))+1)|\N_{\Aut(G)\cap \PGL(V))}(\langle x \rangle)|\\
  &\leq 4|\N_{\Aut(G)\cap \PGL(V))}(\langle x \rangle)|\\
  &\leq 4|\N_{G}(\langle x \rangle)||\Aut(G)\cap\PGL(V)|/|G|.
\end{align*}
Then based on the values of $|\N_{G}(\langle x \rangle)|$ as listed in Table \ref{IU} (which are extracted from \cite[Table 2]{Xia20202}), we obtain that $|\Inv(x)|\leq u(G)$.
\end{proof}

\begin{lemma}
Let $G$, $q$ and $x$ be as in Theorem \ref{T0}. If $q\geq 4$ for $G=\PSL_4(q)$ and $q\geq 2$ for any other $G$, then for a random involution $y$ of $G$, the probability $P(x)$ of $\Cay(G,\{x,x^{-1},y\})$ being a GRR of $G$ is given by
\begin{equation}
\label{e0}
P(x)=\dfrac{|K(x)\cap L(x)|}{i_2(G)}.
\end{equation}
Moreover,
\begin{equation}
\label{e03}
P(x) \geq 1-\frac{i_2\left(\cup_{M\in\mathcal{M}(x)}M\right)}{i(G)}-\frac{\sum_{\alpha\in \Inv(x)}i_2(\C_G(\alpha))}{i(G)}.
\end{equation}
\end{lemma}

\begin{proof}
As a consequence of \cite[Theorem 3]{Godsil1983}), we known that if $G$ has no proper subgroup of index at most $47$, then for an involution $y$ in $G$, the Cayley graph $\Cay(G,\{x,x^{-1},y\})$ is a GRR of $G$ if and only if $G=\langle x,y\rangle$ and $\Aut(G,\{x,x^{-1},y\})=1$. By \cite[Table 4]{Guest2015}, we know that $G$ has no proper subgroup of such small indices. Thus for $y \in \I_2(G)$, the graph $\Cay(G,\{x,x^{-1},y\})$ is a GRR of $G$ if and only if $y\in K(x)\cap L(x)$. This implies that for a random involution $y$ of $G$, the probability $P(x)$ of $\Cay(G,\{x,x^{-1},y\})$ being a GRR of $G$ is given by (\ref{e0}).
Note that
\[
\I_2(G)\setminus K(x)=\{y \in \I_2(G): G\neq\langle x,y\rangle\}=\I_2\left(\cup_{M\in\mathcal{M}(x)}M\right).
\]
Thus
\begin{equation}\label{e01}
|K(x)|=i_2(G)-i_2\left(\cup_{M\in\mathcal{M}(x)}M\right).
\end{equation}
By the definitions of $K(x)$ and $L(x)$, for each $y\in K(x)\setminus L(x)$, there exists a nontrivial automorphism $\alpha\in \Aut(G)$ which fixes $\{x,x^{-1},y\}$ setwise. Note that both $x$ and $x^{-1}$ have order $r>2$ while $y$ is an involution. Thus $x^{\alpha}=x^{-1}$ and $y^{\alpha}=y$. Since $\alpha^2$ fixes both $x$ and $y$, and $\langle x,y\rangle=G$, we have $|\alpha|=2$. This combined with $x^{\alpha}=x^{-1}$ implies $\alpha\in \Inv(x)$. Moreover, by $y^{\alpha}=y$, we have $y\in {\C_G(\alpha)}$. Hence
\begin{equation*}\label{e02}
|K(x)\setminus L(x)|\leq\sum_{\alpha\in \Inv(x)}i_2(\C_G(\alpha)).
\end{equation*}
Combining this with (\ref{e0}) and (\ref{e01}), we obtain that
\begin{align*}
P(x) & =  \frac{|K(x)| - |K(x)\setminus L(x)|}{i_2(G)} \\
& \geq  1-\frac{i_2\left(\cup_{M\in\mathcal{M}(x)}M\right)}{i_2(G)}-\frac{\sum_{\alpha\in \Inv(x)}i_2(\C_G(\alpha))}{i_2(G)}.
\end{align*}
This together with \eqref{eq:iG} implies \eqref{e03}.
\end{proof}

The following approach will be used in our proof of Theorems \ref{T0} and \ref{T}. Given $G$ and $x$ as in Table \ref{Ge}, we will establish an upper bound on $i_2\left(\cup_{M\in\mathcal{M}(x)}M\right)$ and an upper bound on $i_2(\C_G(\alpha))$ which is independent of $\alpha\in \Inv(x)$. These bounds depend on $q$ only, and let us denote them temporarily by $a(q)$ and $b(q)$, respectively. Using these bounds and \eqref{eq:uG}, we obtain from \eqref{e03} that $P(x) \ge 1-\frac{a(q)}{i(G)}-\frac{u(G)b(q)}{i(G)}$. The most difficult part of our proof is to derive sufficiently small lower bounds $a(q)$ and $b(q)$ such that $\frac{a(q)+u(G)b(q)}{i(G)}$ is less than $1$ for large enough $q$ and approaches $0$ as $q$ goes to infinity. We will achieve this in the next three sections for linear groups, symplectic groups and orthogonal groups, respectively, and thus complete the proof of Theorems \ref{T0} and \ref{T}.

\section{Linear groups}
\label{sec:lin}

\subsection{A lemma}

\begin{lemma}\label{eml}
Let $G=\PSL_{n}(q)$, where $n\geq 4$ is even and $q = 2^f$ for some positive integer $f$, and let $x$ be a fixed element of $G$ with order $r\in \ppd (2,nf)$. If $\alpha$ is an involution in $\PGL_{n}(q)$ such that $x^{\alpha}=x^{-1}$, then
\begin{equation}\label{em}
i_2(\C_G(\alpha))\leq i_2(\SL_{n/2}(q))q^{\frac{n^2}{4}-(n-2)}+q^{\frac{n^2}{4}}.
\end{equation}
\end{lemma}

\begin{proof}
We see from the proof of \cite[Lemma 4.2]{Xia20202} that $\alpha$ is induced by some $\beta\in \GL_{n}(q)$ such that $\beta$ has no eigenspace of dimension larger than $n/2$ over $\mathbb{F}_q$. This means that $\alpha$ is represented by a matrix $j_{n/2}(n)$ of the form
\[
\begin{pmatrix}
    I_{n/2} & O\\
    I_{n/2}& I_{n/2}\\
\end{pmatrix}.
\]
Note that all involutions in $\GL(V)$ that commute with $j_{n/2}(n)$ have the form
\[
\begin{pmatrix}
    X & O \\
    R & X\\
\end{pmatrix},
\]
where $X\in \I_2(\GL_{n/2}(q))$ or $X=I_{n/2}$, and $R$ commutes with $X$.

(i) If $X=I_{n/2}$, then there are at most $q^{\frac{n^2}{4}}$ possible matrices $R$ for $X$.

(ii) If $X$ is an involution in $\GL_{n/2}(q)$, then it lies in a conjugacy class of involutions represented by $g$ in $\SL_{n/2}(q)$, where
\[
g=\begin{pmatrix}
    1 & 0\\
    1 & 1\\
\end{pmatrix}
\]
when $n=4$ and
\[
g=j_{\ell}(n/2)=\begin{pmatrix}
    I_{\ell} & O & O\\
    O & I_{n/2-2\ell} & O\\
    I_{\ell} & O & I_{\ell}
\end{pmatrix},\; 1\leq \ell\leq \frac{n}{4}
\]
when $n \geq 6$. Assume
\[
R=\begin{pmatrix}
    b & d\\
    t & z\\
\end{pmatrix}
\]
for $n=4$ and
\[
R=\begin{pmatrix}
    B & C & D\\
    P & Q & S\\
    T & Y & Z
  \end{pmatrix}
\]
for $n\geq 6$, where $B,C,D,P,Q,S,T,Y,Z$ are matrices of appropriate sizes. By a straightforward computation, we deduce from $Rg=gR$ that
\[
R=\begin{pmatrix}
    b & 0\\
    t & b\\
\end{pmatrix}
\]
for $n=4$ and
\[
R=\begin{pmatrix}
    B & O & O\\
    P & Q & O\\
    T & Y & B
  \end{pmatrix}
\]
for $n\geq 6$. This implies that for $n = 4$ the number of possible matrices $R$ is at most $q^{2}$, and for $n\geq 6$ the number of possible matrices $R$ for each $\ell$ with $1\leq \ell\leq n/4$ is at most $q^{\ell^2+(n/2-\ell)^2}$. Since $\ell^2+(n/2-\ell)^2\leq n^2/4-(n-2)$ for $1\leq \ell\leq n/4$, it follows that there are at most $q^{\frac{n^2}{4}-(n-2)}$ possible matrices $R$ for $X$.

Combining the two cases above, we obtain (\ref{em}).
\end{proof}

As a side remark, we notice that if $n \ge 6$ is even then (\ref{em}) can be improved as
\[
i_2(\C_G(\alpha))\leq \sum\limits_{1\leq \ell\leq n/4} |{j_{\ell}(n/2)}^{\SL_{n/2}(q)}|q^{\ell^2+\left(\frac{n}{2}-\ell\right)^2}+q^{\frac{n^2}{4}},
\]
where ${j_{\ell}(n/2)}^{\SL_{n/2}(q)}$ is the conjugacy class represented by $j_{\ell}(n/2)$ in $\SL_{n/2}(q)$.

\subsection{$G=\PSL_4(q)$, where $q=2^f\geq 4$}

The following result proves Theorems \ref{T0} and \ref{T} for $\PSL_4(q)$.

\begin{prop}\label{PL4}
Let $G=\PSL_4(q)$, where $q=2^f\geq 4$ with $f$ a positive integer, and let $x$ be any fixed element of $G$ with order $r\in \ppd (2,4f)$. Let $P(x)$ be the probability that $\Cay(G,\{x,x^{-1},y\})$ is a GRR of $G$ for a random involution $y$ of $G$. Then $P(x)>1-5q^{-1}$. In particular, if $q\geq 2^{3}$, then $P(x)$ is positive and approaches $1$ as $q$ approaches infinity.
\end{prop}

The rest of this subsection is devoted to the proof of Proposition \ref{PL4} using the notation above. Note that $x$ is contained in some cyclic maximal torus of $G$ with order $(q^4-1)/(q-1)$. We assume $x\in\langle z\rangle\leq G$, where $|z|=(q^4-1)/(q-1)$.

By \cite[Proposition 4.1, Corollary 6.2, Proposition 6.4]{King2017}, we find that for maximal subgroups of $G$ in the Aschbacher classes $\mathcal{C}_1, \ldots ,\mathcal{C}_8$, there are at most one subgroup of type $\GL_{2}(q^2).\Cy_2$ and at most $q+1$ subgroups of type $\PSp_{4}(q)$ containing $x$. According to \cite[Table 8.9]{Bray2013}, there is no maximal subgroup of $G$ in class $\mathcal{S}$. Hence, according to the upper bound on $i_2(\GL_{2}(q^2).\Cy_2)$ given in \cite[Proposition 5.1]{King2017} and the bound $i_2(\PSp_{4}(q))<q^6+q^4$ from \cite[Theorem 4.3]{Fulman2017}, we obtain
\begin{align*}
i_2\left(\cup_{M\in\mathcal{M}(x)}M\right) & \leq  i_2(\GL_{2}(q^2).\Cy_2)+(q+1)i_2(\PSp_{4}(q))\\
& \leq  2q^2(q^4-1)/(q-1)+(q+1)(q^6+q^4)\\
& <  3q^7/2.
\end{align*}
Since $i(G)=q^5(q^3-1)$ by Table \ref{IU}, it follows that
\begin{equation}\label{e1}
\frac{i_2\left(\cup_{M\in\mathcal{M}(x)}M\right)}{i(G)}<\frac{3q^7/2}{q^5(q^3-1)}.
\end{equation}

Note that
\[
\N_{\Aut(G)}(\langle x \rangle)=\N_{\Aut(G)}(\langle z \rangle)=\N_{G}(\langle z \rangle).\Out(G).
\]
Set $N = \N_{\Aut(G)}(\langle z \rangle)$. Then $N=(\langle z \rangle \rtimes\langle \xi\rangle). (\langle\eta\rangle\times\langle\gamma\rangle)$, where $\xi$ can be viewed as the Frobenius automorphism of $\F_{q^{4}}$ with order $4$ which maps $z$ to $z^q$, $\eta$ is the field automorphism induced by the Frobenius automorphism of $\F_{q}$ with order $f$, and $\gamma$ is the inverse-transpose automorphism.

Let $u=\diag (\lambda,\lambda^{q},\lambda^{q^2},\lambda^{q^3})$, where $\lambda$ is a primitive $((q^4-1)/(q-1))$th root of unity. By \cite[Theorem 2.12]{Hestenes1970}, we know that $\langle z \rangle$ is conjugate to $\langle u \rangle$ and can be written as $g^{-1}\langle u \rangle g$ for some $g\in G$ such that $z=g^{-1}ug$. Since $u^{\eta}=u^2$ and $u^{\gamma}=u^{-1}$,
we have
\[
z^{\eta^{g}}=z^{g^{-1}\eta g}=g^{-1}(gzg^{-1})^{\eta}g=g^{-1}u^{2}g=z^{2}
\]
and
\[
z^{\gamma^{g}}=z^{g^{-1}\gamma g}=g^{-1}(gzg^{-1})^{\gamma}g=g^{-1}u^{-1}g=z^{-1}.
\]
Note that
\[
C_{(q^4-1)/(q-1)}\rtimes C_{4f}\cong\GAL_1(q^4)/\Z (\GL_4(q))<\PGAL_4(q)<\langle\PGAL_4(q), \gamma\rangle=\Aut(G).
\]
Thus we can write $N$ as $\langle z \rangle. (\langle\varphi\rangle\times\langle\delta\rangle)$, where $\varphi$ is the Frobenius automorphism of $\F_{q^{4}}$ with order $4f$ which maps $z$ to $z^2$ and $\delta$ is a graph automorphism conjugate to $\gamma$ with $z^{\delta}=z^{-1}$.

\begin{lemma}\label{icl4}
If $\alpha\in \Inv(x)$, then $i_2(\C_G(\alpha))<2q^4$.
\end{lemma}

\begin{proof}
Recall that $\Inv(x)=\{\alpha \in \I_2(\Aut(G)): x^{\alpha}=x^{-1}\}$. We know from \cite[Proposition 3.2.11 (iii)]{Burness2016} that $\C_G(\delta)\cong\C_G(\gamma)\cong\C_{\Sp_{4}(q)}(t)$ where $t\in\Sp_{4}(q)$ is a transvection. Thus, according to \cite[Section 4]{Leemans2017}, we know that $\C_G(\delta)$ is of the form $[q^3]\Sp_2(q)$ and $i_2(\C_G(\delta))<2q^4$.

Note that $\langle z \rangle=\C_{\Aut(G)}(z)$ and $\langle z \rangle\lhd N$. Define
$$
\bar{} :  N \rightarrow \langle\varphi\rangle\times\langle\delta\rangle,\; \alpha \mapsto  \langle z \rangle\alpha.
$$
For $\alpha\in \Inv(x)\subseteq N\setminus \langle z \rangle$, since $(\langle z \rangle\alpha)^2=\langle z \rangle$ and $\alpha\notin \langle z \rangle$, we have $|\bar{\alpha}|=2$.

Consider the case $\bar{\alpha}\in\langle\delta\rangle$ first. In this case we have $\alpha=z^{k}\delta$ for some integer $k$. Note that
\[
\delta^{z^{j}}=z^{-j}\delta z^{j}=z^{-j}(\delta z^{j}\delta)\delta=z^{-j}(z^{\delta})^{j}\delta=z^{-2j}\delta.
\]
Thus $\alpha$ is conjugate to $\delta$ by some $z^{j}$ such that $-2j\equiv k\pmod {\frac{q^4-1}{q-1}}$. (Such an integer $j$ exists as $\frac{q^4-1}{q-1}$ is odd.) Hence
\[
i_2(\C_G(\alpha))=i_2(\C_G(\delta))<2q^4.
\]

Next we consider the case $\bar{\alpha}\in\langle\varphi\rangle$. In this case we have $\bar{\alpha}=\varphi^{2f}$ since $|\bar{\alpha}|=2$, and so $\alpha\in\Inn(G)$. Thus, by a straightforward computation using Lemma \ref{eml}, we obtain
\[
i_2(\C_G(\alpha))\leq(q^2-1)q^2+q^4<2q^4.
\]

Finally, in the case when $\bar{\alpha}\notin\langle\varphi\rangle\cup\langle\delta\rangle$, we have $\bar{\alpha}=\varphi^{2f}\delta$ since $|\bar{\alpha}|=2$. Note that $|x|\mid(q^2+1)$ and $z^{\delta}=z^{-1}$. Thus
\[
x^{\alpha}=x^{\varphi^{2f}\delta}=x^{\varphi^{2f}\delta}=(x^{q^2})^{\delta}=(x^{-1})^{\delta}=x,
\]
which leads to a contradiction.
\end{proof}

Now we are ready to prove Proposition \ref{PL4}. According to the proof above, we have
\[
\Inv(x)\subseteq \langle z \rangle\delta\cup\langle z \rangle\varphi^{2f}.
\]
Note that for $\alpha=z^{k}\varphi^{2f}$ of order $2$, where $k$ is an integer, we have
\[
(z^{k}\varphi^{2f})^2=z^{k}(\varphi^{2f}z^{k}\varphi^{2f})=z^{k(q^2+1)}=1.
\]
This means that $(q+1)\mid k$, and so $|\Inv(x)\cap\langle z \rangle\varphi^{2f}|\leq q^2+1$.
Compared with the bound on $|\Inv(x)|$ in Table \ref{IU}, we obtain the following improved bound:
\[
|\Inv(x)|\leq |z|+q^2+1=(q^2+1)(q+2)\leq 13q^3/8.
\]
Combining this with Lemma \ref{icl4} and the value of $i(G)$ in Table \ref{IU}, we have
\begin{equation*}
\frac{\sum_{\alpha\in \Inv(x)}i_2(\C_G(\alpha))}{i(G)}<\frac{13q^3/8\cdot2q^4}{q^5(q^3-1)}=\frac{13q^{7}/4}{q^5(q^3-1)}.
\end{equation*}
This together with (\ref{e03}) and (\ref{e1}) implies
\[
P(x) > 1-\frac{3q^{7}/2+13q^{7}/4}{q^5(q^3-1)} > 1-5q^{-1},
\]
as stated in Proposition \ref{PL4}.

\begin{remark}
For $G=\PSL_4(2)\cong A_8$, (\ref{e0}) may fail. However, we can see from \cite[Theorem 1.3]{Fang2002} and \cite[Section 5]{Leemans2017} that $G$ has a cubic GRR with connection set $\{z,z^{-1},y\}$, where $y$ is an involution and $|z|=7$, though $7\notin \ppd (2,4)$.
\end{remark}

\subsection{$G=\PSL_{n}(q)$ with $n=6\text{ or }8$, where $q=2^f$}

In this subsection we prove the following proposition which confirms Theorems \ref{T0} and \ref{T} for $\PSL_6(q)$ and $\PSL_8(q)$.

\begin{prop}\label{PL68}
Let $G=\PSL_{n}(q)$ with $n=6\text{ or }8$, where $q=2^f$ with $f$ a positive integer, and let $x$ be any fixed element of $G$ with order $r\in \ppd (2,nf)$. Let $P(x)$ be the probability that $\Cay(G,\{x,x^{-1},y\})$ is a GRR of $G$ for a random involution $y$ of $G$. Then
\begin{equation}
\label{eq:PL68}
  P(x) > \left\{\begin{array}{ll}
1-12q^{-4}-84q^{-1} &\text{ for } n=6, \\[0.3cm]
1-12q^{-9}-168q^{-5} &\text{ for } n=8.
\end{array}\right.
\end{equation}
In particular, if $q\geq 2^{7}$ when $n=6$ and $q \ge 2^{2}$ when $n=8$, then $P(x)$ is positive and approaches $1$ as $q$ approaches infinity.
\end{prop}

In what follows we use the notation in Proposition \ref{PL68}. According to \cite[Table 8.25, Table 8.45]{Bray2013}, there is no maximal subgroup of $G$ in $\mathcal{S}$. Thus all maximal subgroups containing $x$ lie in $\mathcal{C}_1, \ldots, \mathcal{C}_8$. By \cite[Proposition 4.1, Corollary 6.2, Proposition 6.4]{King2017}, we know that $\PSL_{6}(q)$ has at most one subgroup of type $\GL_{2}(q^3).\Cy_3$  containing $x$, at most one subgroup of type $\GL_{3}(q^2).\Cy_2$  containing $x$, and at most $(q^3-1)/(q-1)$ subgroups of type $\PSp_{6}(q)$ containing $x$, and that $\PSL_{8}(q)$ has at most one subgroup of type $\GL_{4}(q^2).\Cy_2$  containing $x$ and at most $(q^4-1)/(q-1)$ subgroups of type $\PSp_{8}(q)$ containing $x$. Hence, by the upper bounds on the number of involutions in maximal subgroups of geometric type containing $x$ as given in \cite[Proposition 5.1]{King2017}, we have
\begin{align*}
i_2\left(\cup_{M\in\mathcal{M}(x)}M\right) & \leq  i_2(\GL_{2}(q^3).\Cy_3)+i_2(\GL_{3}(q^2).\Cy_2)+\frac{q^3-1}{q-1}i_2(\PSp_{6}(q))\\
& \leq  2\frac{q^6-1}{q-1}q^3+2\frac{q^4-1}{q-1}q^8+\frac{q^3-1}{q-1}\cdot 2(q+1)q^{11}\\
& <  6q^{14}
\end{align*}
for $n=6$ and
\begin{align*}
i_2\left(\cup_{M\in\mathcal{M}(x)}M\right) & \leq  i_2(\GL_{4}(q^2).\Cy_2)+\frac{q^4-1}{q-1}i_2(\PSp_{8}(q))\\
& \leq  2\frac{q^4-1}{q-1}q^{16}+\frac{q^4-1}{q-1}\cdot 2(q+1)q^{19}\\
& <  6q^{23}
\end{align*}
for $n=8$. Since $i(G)=q^{\frac{n^2}{2}}/2$ by Table \ref{IU}, it follows that
\begin{equation}\label{gl68}
\frac{i_2\left(\cup_{M\in\mathcal{M}(x)}M\right)}{i(G)} < \left\{\begin{array}{ll}
12q^{-4} &\text{ for } n=6, \\[0.3cm]
12q^{-9} &\text{ for } n=8.
\end{array}\right.
\end{equation}

\begin{lemma}
\label{lem:psl}
If $\alpha\in \Inv(x)$, then $i_2(\C_G(\alpha))<\dfrac{21}{16}q^{\frac{n^2}{4}+\frac{n}{2}}$.
\end{lemma}
\begin{proof}
Let $\alpha\in \Inv(x)$. If $\alpha\in \Aut(G)\setminus\PGL(V)$, then by \cite[Proposition 4.4]{Liebeck1996}, we know that $\C_G(\alpha)$ is of type $\PSL_{n}(q^{\frac{1}{2}})$ or $\PSU_{n}(q^{\frac{1}{2}})$, or $\C_G(\alpha)\leq\Sp_{n}(q)$. If $\C_G(\alpha)$ is of type $\PSL_{n}(q^{\frac{1}{2}})$ or $\PSU_{n}(q^{\frac{1}{2}})$, then
\[
i_2(\C_G(\alpha))<3q^{\frac{n^2}{4}+\frac{n}{4}-\frac{1}{2}}<\frac{21}{16}q^{\frac{n^2}{4}+\frac{n}{2}}
\]
by \cite[Lemma 2.6]{Xia20202}. If $\C_G(\alpha)\leq\Sp_{n}(q)$, then
\[
i_2(\C_G(\alpha))<\frac{21}{16}q^{\frac{n^2}{4}+\frac{n}{2}}
\]
by \cite[Theorem 4.3]{Fulman2017}. If $\alpha\in \Aut(G)\cap\PGL(V)=\PGL_n(q)$, then by a straightforward computation using Lemma \ref{eml} we obtain
$$
i_2(\C_G(\alpha))\leq  2q^4\cdot q^5+q^9 < 3q^{9}< \frac{21}{16}q^{12}
$$
for $G=\PSL_{6}(q)$ and
$$
i_2(\C_G(\alpha))\leq  2q^8\cdot q^{10}+q^{16} < 3q^{18}< \frac{21}{16}q^{20}
$$
for $G=\PSL_{8}(q)$.
\end{proof}

Combining Lemma \ref{lem:psl} with the values of $i(G)$ and $u(G)$ in Table \ref{IU}, we obtain
$$
\frac{\sum_{\alpha\in \Inv(x)}i_2(\C_G(\alpha))}{i(G)} < \left\{\begin{array}{ll}
84q^{-1} &\text{ for } n=6, \\[0.3cm]
168q^{-5} &\text{ for } n=8.
\end{array}\right.
$$
This together with (\ref{e03}) and (\ref{gl68}) implies \eqref{eq:PL68}, completing the proof of Proposition \ref{PL68}.

\begin{remark}
By \cite[Theorem 1.1]{Spiga2018}, for $G=\PSL_6(q)$ with $\gcd(6,q-1)=3$, a cubic Cayley graph of $G$ with connection set $S$ is a GRR of $G$ if and only if $\langle S\rangle=G$ and $\Aut(G,S)=1$. (Here we do not require $S$ to contain an element of odd prime order.) Combining this with \cite[Theorem 1.1]{Leemans2017}, we see that $G$ can be proved to have cubic GRRs. The proof involves a connection set $\{z,z^{-1},y\}$, where $z$ is an element of order $(q^6-1)(q-1)^{-1}/3$ and $y$ is an involution. Thus $\PSL_{6}(q)$ with $q=2^f\leq 2^{6}$ has a cubic GRR when $f$ is even.
\end{remark}

\section{Symplectic groups}
\label{sec:sym}

In this section we prove the following proposition which confirms Theorems \ref{T0} and \ref{T} for $\PSp_6(q)$ and $\PSp_8(q)$.

\begin{prop}\label{PS68}
Let $G=\PSp_{n}(q)$ with $n=6\text{ or }8$, where $q=2^f$ with $f$ a positive integer, and let $x$ be any fixed element of $G$ with order $r\in \ppd (2,nf)$. Let $P(x)$ be the probability that $\Cay(G,\{x,x^{-1},y\})$ is a GRR of $G$ for a random involution $y$ of $G$. Then
\begin{equation}
\label{eq:PS68}
  P(x) > \left\{\begin{array}{ll}
1-57q^{-1} &\text{ for } n=6, \\[0.3cm]
1-8q^{-4}-68q^{-2} &\text{ for } n=8.
\end{array}\right.
\end{equation}
In particular, if $q\geq 2^{6}$ when $n = 6$ and $q \ge 2^{4}$ when $n = 8$, then $P(x)$ is positive and approaches $1$ as $q$ approaches infinity.
\end{prop}

In what follows we use the notation in Proposition \ref{PS68}. Since $\ppd (2,6) = \emptyset$, we may assume $q\geq 4$ in $\PSp_{6}(q)$. We first consider maximal subgroups containing $x$ in $\mathcal{C}_1, \ldots, \mathcal{C}_8$. By \cite[Proposition 4.1, Corollary 6.2, Proposition 6.4]{King2017}, we know that $\PSp_{6}(q)$ has at most one subgroup of type $\Sp_{2}(q^3).\Cy_3$ containing $x$ and at most two subgroups of type $\PSO^{-}_{6}(q)$ containing $x$, and that $\PSp_{8}(q)$ has at most one subgroup of type $\Sp_{4}(q^2).\Cy_2$ containing $x$ and at most two subgroups of type $\PSO^{-}_{8}(q)$ containing $x$. For maximal subgroups of $G$ in $\mathcal{S}$, according to \cite[Table 8.29, Table 8.49]{Bray2013}, \cite[Corollary 6.2]{King2017} and the value of $|\N_{G}(\langle x \rangle)|$ in \cite[Table 9]{King2017}, we know that $\PSp_{6}(q)$ has at most $6(q^3+1)$ subgroups of type $\G_{2}(q)$ containing $x$, and that $\PSp_{8}(q)$ has no subgroup in $\mathcal{S}$ containing $x$ when $q\geq 4$ and at most $16(q^4+1)/r$ subgroups of type $\PSL_{2}(17)$ containing $x$ when $q=2$ and $r=17$. Hence, by the upper bounds on the number of involutions in maximal subgroups of geometric type containing $x$ proved in \cite[Proposition 5.1]{King2017} and the bound $i_2(\G_{2}(q))<q^8+q^6$ given in \cite[Table 5]{Burness2018}, we have
\begin{align*}
i_2\left(\cup_{M\in\mathcal{M}(x)}M\right) & \leq  i_2(\Sp_2(q^3).\Cy_3)+2i_2(\PSO^{-}_{6}(q))+6(q^3+1)i_2(\G_{2}(q))\\
& <  2(q^3+1)q^3+4(q+1)q^{8}+6(q^3+1)(q^8+q^6)\\
& <  7q^{11}
\end{align*}
for $n=6$ and
\begin{align*}
i_2\left(\cup_{M\in\mathcal{M}(x)}M\right) & \leq  i_2(\Sp_4(q^2).\Cy_2)+2i_2(\PSO^{-}_{8}(q))+16i_2(\PSL_{2}(17))\\
& \leq  2(q^2+1)q^{10}+4(q+1)q^{15}+16\cdot 153\\
& <  7q^{16}+2448\\
& <  8q^{16}
\end{align*}
for $n=8$. Since $i(G)=q^{\frac{n^2}{4}+\frac{n}{2}}$ by Table \ref{IU}, it follows that
\begin{equation}\label{gp68}
\frac{i_2\left(\cup_{M\in\mathcal{M}(x)}M\right)}{i(G)} < \left\{\begin{array}{ll}
7q^{-1} &\text{ for } n=6, \\[0.3cm]
8q^{-4} &\text{ for } n=8.
\end{array}\right.
\end{equation}

\begin{lemma}\label{ics68}
If $\alpha\in \Inv(x)$, then
\begin{equation}
\label{eq:ics68}
i_2(\C_G(\alpha)) \leq \left\{\begin{array}{ll}
2q^{8} &\text{ for } n=6, \\[0.3cm]
2q^{14} &\text{ for } n=8.
\end{array}\right.
\end{equation}
\end{lemma}

\begin{proof}
Let $\alpha\in \Inv(x)$. If $\alpha\in \Aut(G)\setminus\PGL(V)$, then by \cite[(19.2)]{Aschbacher1976} and \cite[Proposition 4.4]{Liebeck1996}, we know that $\alpha$ is a field automorphism and that $\C_G(\alpha)$ is of type $\PSp_{n}(q^{\frac{1}{2}})$. So it follows from \cite[Lemma 2.6]{Xia20202} that $i_2(\C_G(\alpha))<3q^{\frac{n^2}{8}+\frac{n}{4}}$. This inequality implies \eqref{eq:ics68} in the case when $n=6$ or $8$.

In the rest of the proof we assume that $\alpha\in \Aut(G)\cap\PGL(V)$. In this case, by \cite[Table 12]{Xia20202}, we know that $\C_G(\alpha)$ is contained in a maximal parabolic subgroup $M=\PR_{n/2}$ of $G$. We may write $\PR_{n/2}=QL$, where $|Q|=q^{\frac{n(n+2)}{8}}$ and $L\cong\GL_{n/2}(q)$. Similar to the proof of \cite[Lemma 5.12]{Liebeck1996}, we have
\begin{equation}
\label{eq:i2M}
i_2(M)=i_2(Q)+\sum\limits_{g \in \I_2(L)}|\C^*_Q(g)|,
\end{equation}
where $\C^*_Q(g)=\{u\in Q: gug=u^{-1}\}$ for $g \in \I_2(L)$.

Case 1: $n = 6$.

In this case we have $M=\PR_3=QL\leq\PSp_{6}(q)$, where $|Q|=q^{6}$ and $L\cong\GL_3(q)$, and the involution class in $L$ is represented by the matrix
\[
x_{1}=\left(
  \begin{array}{cc}
    j_{1}(3) & O \\
    O & j_{1}(3)^{\T}
  \end{array}
\right).
\]
Recall that $j_{1}(3)$ is the involution of the Suzuki form in $\GL_3(q)$. By the upper bound(s) on $|\C^*_Q(g)|$ for involutions in distinct conjugacy class(es) in $L$ given in the eleventh line on page 109 in the proof of \cite[Lemma 5.12]{Liebeck1996}, we obtain that $|\C^*_Q(x_{1})|\leq q^{4}$. On the other hand, since $L\cong \GL_3(q)$, we have $i_2(L)\leq 3q^4/2$. Thus, by \eqref{eq:i2M}, we have
\[
i_2(M)\leq (q^{6}-1)+3q^{8}/2\leq 2q^{8}.
\]

Case 2: $n = 8$.

In this case we have $M=\PR_4=QL\leq\PSp_{8}(q)$, where $|Q|=q^{10}$ and $L\cong\GL_4(q)$, and the involution classes in $L$ are represented by the matrices $x_{r}=x_{1}, x_{2}$, where
\[
x_{r}=\left(
  \begin{array}{cc}
    j_{r}(4) & O \\
    O & j_{r}(4)^{\T}
  \end{array}
\right).
\]
Similar to Case 1, by the above-mentioned upper bounds on $|\C^*_Q(g)|$ produced in the proof of \cite[Lemma 5.12]{Liebeck1996}, we obtain
\begin{align*}
|\C^*_Q(x_{r})|\leq\left\{\begin{array}{ll}
q^{7} &\text{ for } r=1, \\[0.3cm]
q^{6} &\text{ for } r=2.
\end{array}\right.
\end{align*}
Computing the length of each class of involutions in $\GL_4(q)$ by (\ref{cis}) as in the proof of Lemma \ref{lem:24}, we obtain $ |x_{1}^{L}|\leq 7q^6/4$ and $|x_{2}^{L}|\leq q^8$. Thus, by \eqref{eq:i2M}, we have
\[
i_2(M)\leq (q^{10}-1)+7q^{13}/4+q^{14}\leq 2q^{14}.
\]

In summary, in the case when $\alpha\in \Aut(G)\cap\PGL(V)$, we have proved that $i_2(M)\leq 2q^{8}$ for $n=6$ and $i_2(M)\leq 2q^{14}$ for $n=6$. Since $\C_G(\alpha)$ is contained in $M$, we have $i_2(\C_G(\alpha))\leq i_2(M)$ and hence \eqref{eq:ics68} follows immediately.
\end{proof}

Combining Lemma \ref{ics68} with the values of $i(G)$ and $u(G)$ in Table \ref{IU}, we obtain
\begin{equation*}
\frac{\sum_{\alpha\in \Inv(x)}i_2(\C_G(\alpha))}{i(G)} \leq \left\{\begin{array}{ll}
50q^{-1} &\text{ for } n=6, \\[0.3cm]
68q^{-2} &\text{ for } n=8.
\end{array}\right.
\end{equation*}
This together with (\ref{e03}) and (\ref{gp68}) implies \eqref{eq:PS68}, completing the proof of Proposition \ref{PS68}.

\section{Orthogonal groups}
\label{sec:ortho}

\subsection{$G=\PR\Omega_{n}^{+}(q)$ with $n=8, 10\text{ or }12$, where $q=2^f$}

In this subsection we prove the following result which confirms Theorems \ref{T0} and \ref{T} for $\PR\Omega_{8}^{+}(q)$, $\PR\Omega_{10}^{+}(q)$ and $\PR\Omega_{12}^{+}(q)$.

\begin{prop}\label{PO+}
Let $G=\PR\Omega_{n}^{+}(q)$ with $n=8, 10\text{ or }12$, where $q=2^f$ with $f$ a positive integer, and let $x$ be any fixed element of $G$ with order $r\in \ppd (2,(n-2)f)$. Let $P(x)$ be the probability that $\Cay(G,\{x,x^{-1},y\})$ is a GRR of $G$ for a random involution $y$ of $G$. Then
\begin{equation}
\label{eq:PO+}
  P(x) \geq \left\{\begin{array}{ll}
1-24q^{-3}-244q^{-2} &\text{ for } n=8, \\[0.3cm]
1-20q^{-3}-612q^{-2} &\text{ for } n=10, \\[0.3cm]
1-20q^{-5}-744q^{-6} &\text{ for } n=12.
\end{array}\right.
\end{equation}
In particular, if $q\geq 2^{4}, 2^{5}, 2^{2}$ when $n=8, 10, 12$, respectively, then $P(x)$ is positive and approaches $1$ as $q$ approaches infinity.
\end{prop}

We use the notation in Proposition \ref{PO+} in the following proof of this result. Since $\ppd (2,6) = \emptyset$, we may assume $q\geq 4$ in $\PR\Omega_{8}^{+}(q)$. As shown in Table \ref{O+}, for every maximal subgroup $M$ of $G$ containing $x$, by \cite[Tables 6, 10]{King2017}, \cite[Table 2]{Xia20202} and \cite[Proposition 2.3]{King2017}, we know the number $c_M$ of $G$-conjugacy classes of each type, an upper bound $m_M$ on $|\N_G(\langle x \rangle)|/|\N_M(\langle x \rangle)|$ and an upper bound $i_M$ on $i_2(M)$. This together with \cite[Corollary 6.2]{King2017} implies
\begin{equation*}
i_2\left(\cup_{M\in\mathcal{M}(x)}M\right)\leq \left\{\begin{array}{ll}
12q^{13} &\text{ for } n=8 , \\[0.3cm]
10q^{21} &\text{ for } n=10 , \\[0.3cm]
10q^{31} &\text{ for } n=12 .
\end{array}\right.
\end{equation*}
Combining this with the value of $i(G)$ in Table \ref{IU}, we obtain
\begin{equation}\label{go1+}
\frac{i_2\left(\cup_{M\in\mathcal{M}(x)}M\right)}{i(G)} \leq
\left\{\begin{array}{ll}
24q^{-3} &\text{ for } n=8 , \\[0.3cm]
20q^{-3} &\text{ for } n=10 , \\[0.3cm]
20q^{-5} &\text{ for } n=12 .
\end{array}\right.
\end{equation}

\begin{table}[t]
\centering
		\begin{tabular}{l|l|l|l|l|l}
			\hline
			$G$ & Type of $M$ & $c_M$ & $m_M$ & $i_M$ & Conditions\\ \hline
            \multirow{5}*{$\PR\Omega_{8}^{+}(q)$}& $\rmO_6^{-}(q)\times \rmO_2^{-}(q)$ & $1$ &  $1$ & $4q^{10}$ & ~\\
            ~ & $\rmO_7(q)$ & $1$ & $q+1$ & $3q^{12}$ & ~\\
            ~ & $\GU_4(q).\Cy_2$ & $2$ & $1$ &  $2q^{10}$ & ~\\
            ~ & $\Soc(M)=\PSp_{6}(q)$ & $2$ &  $q+1$ & $3q^{12}$ & ~\\
            ~ & $\Soc(M)=\PSU_3(q)$ & $1$  &  $2(q+1)^2$ & $3q^{5}$ & $q\equiv 2\pmod 3$ and $q\neq 2$\\\hline
            \multirow{2}*{$\PR\Omega_{10}^{+}(q)$}& $\rmO_8^{-}(q)\times \rmO_2^{-}(q)$ & $1$ & $1$ & $5q^{17}$ & ~\\
            ~ & $\rmO_9(q)$ & $1$ & $2(q+1)$ & $3q^{20}$ & ~\\ \hline
            \multirow{3}*{$\PR\Omega_{12}^{+}(q)$}& $\rmO_{10}^{-}(q)\times \rmO_2^{-}(q)$ & $1$ & $1$ & $5q^{26}$ & ~\\
            ~ & $\rmO_{11}(q)$ & $1$ & $2(q+1)$ & $3q^{30}$ & ~\\
            ~ & $\GU_6(q).\Cy_2$ & $2$ & $1$ & $5q^{21}$ & ~\\ \hline
		\end{tabular}
\caption{The tuples $(M, c_M, m_M, i_M)$ for some $\PR\Omega_{n}^{+}(q)$ }
\label{O+}
\end{table}

\begin{lemma}\label{ico+}
If $\alpha\in \Inv(x)$, then
\begin{equation}\label{to1+0}
i_2(\C_G(\alpha)) \leq \left\{\begin{array}{ll}
2q^{10} &\text{ for } n=8, \\[0.3cm]
3q^{17} &\text{ for } n=10, \\[0.3cm]
3q^{24} &\text{ for } n=12.
\end{array}\right.
\end{equation}
\end{lemma}

\begin{proof}
Let $\alpha\in \Inv(x)$. If $\alpha\in \Aut(G)\setminus\PGL(V)$, then by \cite[Lemma 4.1]{Xia20202}, we have
\begin{equation*}\label{to1+01}
i_2(\C_G(\alpha))<3q^{n^2/8},
\end{equation*}
which implies \eqref{to1+0}.

In the remainder of the proof we assume that $\alpha\in \Aut(G)\cap\PGL(V)$. In this case, by \cite[Table 13]{Xia20202}, we know that $\C_G(\alpha)$ is contained in a maximal subgroup $M=\PR_{m}$ of $G$, where $m=n/2-1\text{ or }n/2$. We may write $M=QL$, where $|Q|=q^{m(n-3m/2-1/2)}$ and $L\lesssim\GL_m(q)\times \Om_{n-2m}^{+}(q)$. Similar to the proof of \cite[Lemma 5.12]{Liebeck1996}, we have
\begin{equation}\label{iMP}
i_2(M)=i_2(Q)+\sum\limits_{g \in \I_2(L)}|\C^*_Q(g)|,
\end{equation}
where $\C^*_Q(g)=\{u\in Q: gug=u^{-1}\}$ for $g \in \I_2(L)$. Since $q$ is even, we have $\Om_{2}^{+}(q)\cong C_{q-1}$  and hence $i_2(\Om_{2}^{+}(q))=0$. Since $\C_G(\alpha)$ is contained in $M$, we have $i_2(\C_G(\alpha)) \leq i_2(M)$. Thus, it suffices to prove that the right-hand side of \eqref{to1+0} is an upper bound on $i_2(M)$, and this is what we aim to achieve in the following.

Case 1: $G=\PR\Omega_{8}^{+}(q)$.

In this case we have $m = 3$ or $4$ as $n = 8$.

Subcase 1.1: $m = 3$. In this case we have $M=\PR_3=QL$, where $|Q|=q^{9}$ and $L\lesssim\GL_3(q)\times \Om_{2}^{+}(q)$, and the involution classes in $L$ are represented by the matrix
\[
x_{1}=\left(
  \begin{array}{ccc}
    j_{1}(3) & O & O\\
    O & a_{0}(2) & O\\
    O & O & j_{1}(3)^{\T}
  \end{array}
\right).
\]
Here $j_{1}(3)$ is again the Suzuki form as mentioned before, while $a_{0}(2)$ represents the identity in $\Om_{2}^{+}(q)$. By the upper bound(s) on $|\C^*_Q(g)|$ for involutions in distinct conjugacy class(es) in $L$ as given in the sixth line from the bottom on page 110 in the proof of \cite[Lemma 5.12]{Liebeck1996}, we obtain $|\C^*_Q(x_{1})| \leq q^{5}$. Since $i_2(\GL_3(q))\leq 3q^4/2$, we have $i_2(L)\leq 3q^4/2$. Thus, by \eqref{iMP},
$i_2(M)\leq (q^{9}-1)+3q^{9}/2\leq 3q^{9} < 2q^{10}$ as required.

Subcase 1.2: $m = 4$. In this case we have $M=\PR_4=QL$, where $|Q|=q^{6}$ and $L\lesssim\GL_4(q)$, and the involution classes in $L$ are represented by the matrices $x_{r}=x_{1}, x_{2}$, where
\[
x_{r}=\left(
  \begin{array}{cc}
    j_{r}(4) & O \\
    O & j_{r}(4)^{\T}
  \end{array}
\right).
\]
By the sixth line from the bottom on page 110 in the proof of \cite[Lemma 5.12]{Liebeck1996}, we have
\begin{align*}
|\C^*_Q(x_{r})|\leq\left\{\begin{array}{ll}
q^{3} &\text{ for } r=1, \\[0.3cm]
q^{2} &\text{ for } r=2.
\end{array}\right.
\end{align*}
Computing the length of each class of involutions in $\GL_4(q)$ by (\ref{cis}), we obtain $|x_{1}^{L}|\leq 2q^6$ and $|x_{2}^{L}|\leq q^8$. Plugging these into \eqref{iMP}, we obtain $i_2(M)\leq (q^6-1)+2q^9+q^{10}\leq 2q^{10}$ as required.

Case 2. $G=\PR\Omega_{10}^{+}(q)$.

In this case we have $m = 4$ or $5$ as $n=10$.

Subcase 2.1: $m = 4$. In this case we have $M=\PR_4=QL$, where $|Q|=q^{14}$ and $L\lesssim\GL_4(q)\times \Om_{2}^{+}(q)$, and the involution classes in $L$ are represented by the matrices $x_{r}=x_{1}, x_{2}$, where
\[
x_{r}=\left(
  \begin{array}{ccc}
    j_{r}(4) & O & O\\
    O & a_{0}(2) & O\\
    O & O & j_{r}(4)^{\T}
  \end{array}
\right).
\]
Similar to Case 1, we have
\begin{align*}
|\C^*_Q(x_{r})|\leq \left\{\begin{array}{ll}
q^{9} &\text{ for } r=1, \\[0.3cm]
q^{6} &\text{ for } r=2.
\end{array}\right.
\end{align*}
Considering the length of each class of involutions in $\GL_4(q)$, we obtain $|x_{1}^{L}|\leq 2q^6$ and $|x_{2}^{L}|\leq q^8$. This together with \eqref{iMP} yields $i_2(M)\leq (q^{14}-1)+2q^{15}+q^{14}\leq 3q^{15} < 3q^{17}$ as desired.

Subcase 2.2: $m = 5$. In this case we have $M=\PR_5=QL$, where $|Q|=q^{10}$ and $L\lesssim\GL_5(q)$, and the involution classes in $L$ are represented by the matrices $x_{r}=x_{1}, x_{2}$, where
\[
x_{r}=\left(
  \begin{array}{cc}
    j_{r}(5) & O \\
    O & j_{r}(5)^{\T}
  \end{array}
\right).
\]
Similar to Case 1, we have
\begin{align*}
|\C^*_Q(x_{r})|\leq\left\{\begin{array}{ll}
q^{6} &\text{ for } r=1, \\[0.3cm]
q^{5} &\text{ for } r=2.
\end{array}\right.
\end{align*}
Computing the length of each class of involutions in $\GL_5(q)$ by (\ref{cis}), we obtain
$|x_{1}^{L}|\leq 2q^8$ and $|x_{2}^{L}|\leq 2q^{12}$. Plugging these into \eqref{iMP}, we obtain $i_2(M)\leq (q^{14}-1)+2q^{14}+2q^{17}\leq 3q^{17}$ as desired.

Case 3. $G=\PR\Omega_{12}^{+}(q)$.

In this case we have $m = 5$ or $6$ as $n = 12$.

Subcase 3.1: $m = 5$. In this case we have $M=\PR_5=QL$, where $|Q|=q^{20}$ and $L\lesssim\GL_5(q)\times \Om_{2}^{+}(q)$, and the involution classes in $L$ are represented by the matrices $x_{r}=x_{1}, x_{2}$, where
\[
x_{r}=\left(
  \begin{array}{ccc}
    j_{r}(5) & O & O\\
    O & a_{0}(2) & O\\
    O & O & j_{r}(5)^{\T}
  \end{array}
\right).
\]
Similar to Case 1, we have
\begin{align*}
|\C^*_Q(x_{r})|\leq \left\{\begin{array}{ll}
q^{14} &\text{ for } r=1, \\[0.3cm]
q^{10} &\text{ for } r=2.
\end{array}\right.
\end{align*}
Considering the length of each class of involutions in $\GL_5(q)$, we obtain $|x_{1}^{L}|\leq 2q^8$ and $|x_{2}^{L}|\leq 2q^{12}$. Combining these with \eqref{iMP}, we obtain $i_2(M)\leq (q^{20}-1)+2q^{22}+2q^{22}\leq 5q^{22} < 3q^{24}$ as required.

Subcase 3.2: $m = 6$. In this case we have $M=\PR_6=QL$, where $|Q|=q^{15}$ and $L\lesssim\GL_6(q)$, and the involution classes in $L$ are represented by the matrices $x_{r}=x_{1}, x_{2}, x_{3}$, where
\[
x_{r}=\left(
  \begin{array}{cc}
    j_{r}(6) & O \\
    O & j_{r}(6)^{\T}
  \end{array}
\right).
\]
Similar to Case 1, we have
\begin{align*}
|\C^*_Q(x_{r})|\leq\left\{\begin{array}{ll}
q^{10} &\text{ for } r=1, \\[0.3cm]
q^{7} &\text{ for } r=2, \\[0.3cm]
q^{6} &\text{ for } r=3.
\end{array}\right.
\end{align*}
Computing the length of each class of involutions in $\GL_6(q)$ by (\ref{cis}), we obtain $|x_{1}^{L}|\leq 2q^{10}$, $|x_{2}^{L}|\leq 3q^{16}$ and $|x_{3}^{L}|\leq q^{18}$. Combining these with \eqref{iMP}, we obtain $i_2(M)\leq (q^{15}-1)+2q^{20}+3q^{23}+q^{24}\leq 3q^{24}$ as required.
\end{proof}

Combining Lemma \ref{ico+} with the values of $i(G)$ and $u(G)$ in Table \ref{IU}, we obtain
\begin{equation*}
\frac{\sum_{\alpha\in \Inv(x)}i_2(\C_G(\alpha))}{i(G)} \leq \left\{\begin{array}{ll}
244q^{-2} &\text{ for } n=8, \\[0.3cm]
612q^{-2} &\text{ for } n=10, \\[0.3cm]
744q^{-6} &\text{ for } n=12.
\end{array}\right.
\end{equation*}
This together with (\ref{e03}) and (\ref{go1+}) implies \eqref{eq:PO+}, completing the proof of Proposition \ref{PO+}.

\subsection{$G=\PR\Omega_{n}^{-}(q)$ with $n=8, 10\text{ or }12$, where $q=2^f$}

In this subsection we aim to prove the following result which confirms Theorems \ref{T0} and \ref{T} for $\PR\Omega_{8}^{-}(q)$, $\PR\Omega_{10}^{-}(q)$ and $\PR\Omega_{12}^{-}(q)$.

\begin{prop}\label{PO-}
Let $G=\PR\Omega_{n}^{-}(q)$ with $n=8, 10\text{ or }12$, where $q=2^f$ with $f\geq2$, and let $x$ be any fixed element of $G$ with order $r\in \ppd (2,nf)$. Let $P(x)$ be the probability that $\Cay(G,\{x,x^{-1},y\})$ is a GRR of $G$ for a random involution $y$ of $G$. Then
\begin{equation}
\label{eq:PO-}
  P(x) \geq \left\{\begin{array}{ll}
1-6q^{-8}-204q^{-3} &\text{ for } n=8, \\[0.3cm]
1-12q^{-9}-252q^{-4} &\text{ for } n=10, \\[0.3cm]
1-8q^{-18}-490q^{-8} &\text{ for } n=12.
\end{array}\right.
\end{equation}
In particular, if $q\geq 2^{3}, 2^{2}, 2^{2}$ when $n=8, 10, 12$, respectively, then $P(x)$ is positive and approaches $1$ as $q$ approaches infinity.
\end{prop}

In what follows we will use the notation in Proposition \ref{PO-}. Based on \cite[Tables 6, 10]{King2017}, \cite[Table 2]{Xia20202} and \cite[Proposition 2.3]{King2017}, Table \ref{O-} gives the number $c_M$ of $G$-conjugacy classes of each type, an upper bound $m_M$ on $|\N_G(\langle x \rangle)|/|\N_M(\langle x \rangle)|$, and an upper bound $i_M$ on $i_2(M)$, for every maximal subgroup $M$ of $G$ containing $x$. By this table and \cite[Corollary 6.2]{King2017}, we have
\begin{equation*}
i_2\left(\cup_{M\in\mathcal{M}(x)}M\right)\leq \left\{\begin{array}{ll}
3q^{8} &\text{ for } n=8 , \\[0.3cm]
6q^{15} &\text{ for } n=10 , \\[0.3cm]
4q^{18} &\text{ for } n=12 .
\end{array}\right.
\end{equation*}
Combining this with the value of $i(G)$ in Table \ref{IU}, we obtain
\begin{equation}\label{go1-}
\frac{i_2\left(\cup_{M\in\mathcal{M}(x)}M\right)}{i(G)}\leq \left\{\begin{array}{ll}
6q^{-8} &\text{ for } n=8 , \\[0.3cm]
12q^{-9} &\text{ for } n=10 , \\[0.3cm]
8q^{-18} &\text{ for } n=12 .
\end{array}\right.
\end{equation}

\begin{table}[h]
\centering
		\begin{tabular}{l|l|l|l|l}
			\hline
			$G$ & Type of $M$ & $c_M$ & $m_M$ & $i_M$\\ \hline
            \multirow{1}*{$\PR\Omega_{8}^{-}(q)$}& $\rmO_4^{-}(q^2).\Cy_2$ & $1$ & $1$ & $3q^{8}$ \\ \hline
            \multirow{3}*{$\PR\Omega_{10}^{-}(q)$}& $\rmO_5^{-}(q^2).\Cy_2$ & $1$ & $1$ & $3q^{25/2}$ \\
            ~ & $\GU_5(q).\Cy_2$ & $1$ & $1$ & $5q^{15}$ \\
            \hline
            \multirow{3}*{$\PR\Omega_{12}^{-}(q)$}& $\rmO_6^{-}(q^2).\Cy_2$ & $1$ & $1$ & $3q^{18}$ \\
            ~ & $\rmO_4^{-}(q^3).\Cy_3$ & $1$ & $1$ & $3q^{12}$ \\
            \hline
		\end{tabular}
\caption{The tuple $(M, c_M, m_M, i_M)$ for some $\PR\Omega_{n}^{-}(q)$ }
\label{O-}
\end{table}

\begin{remark}
\label{rem:errors}
For $G=\PR\Omega_{n}^{-}(q)$ with $n\geq 8$ even, there are errors in \cite[Table 14]{Xia20202} when considering a maximal subgroup $M$ containing $\C_G(\alpha)$ for an automorphism $\alpha\in \Inv(x)$, where $x$ is a fixed element of $G$ with order $r\in \ppd (2,nf)$. In fact, $\PR_{n/2-1}$ instead of $\PR_{n/2}$ is a possible type, and subgroups of type $\rmO_{n/2}^{+}(q)\times\rmO_{n/2}^{-}(q)$ should be considered when $q$ is odd and $n/2$ is even. (Note that $\PR_{n/2-1}$ stands for the parabolic subgroups which stabilize some subspace of dimension $n/2-1$.) In the first case, we have
\[
|G:M|=\frac{q^{\frac{n}{2}}+1}{q+1}\prod_{i=1}^{n/2-1}(q^{i}+1)>q^{\frac{n^2}{8}-\frac{n}{4}}.
\]
In the second case, we have
\[
|G:M|=q^{\frac{n^2}{8}+1}\prod_{i=1}^{n/4}\frac{(q^{\frac{n}{2}+2i}-1)}{q^{2i}-1}>q^{\frac{n^2}{8}-\frac{n}{4}}.
\]
Thus \cite[Lemma 4.2]{Xia20202} and the main results in \cite{Xia20202} are still valid despite the errors in \cite[Table 14]{Xia20202}.
\end{remark}

Let $\alpha\in \Inv(x)$. Note that, by \cite[Proposition 3.5.25 (ii)]{Burness2016}, there is no such an involutory automorphism outside $\PGL(V)$. By \cite[Table 14]{Xia20202} and Remark \ref{rem:errors}, we know that $\C_G(\alpha)$ is contained in a maximal parabolic subgroup $M=\PR_{n/2-1}$ of $G$. Similar to Lemma \ref{ico+}, we can prove the following lemma.

\begin{lemma}
If $\alpha\in \Inv(x)$, then
\begin{equation*}
i_2(\C_G(\alpha)) \leq \left\{\begin{array}{ll}
3q^{9} &\text{ for } n=8, \\[0.3cm]
3q^{15} &\text{ for } n=10, \\[0.3cm]
5q^{22} &\text{ for } n=12.
\end{array}\right.
\end{equation*}
\end{lemma}

This together with the values of $i(G)$ and $u(G)$ in Table \ref{IU} implies that
\begin{equation*}
\frac{\sum_{\alpha\in \Inv(x)}i_2(\C_G(\alpha))}{i(G)} \leq \left\{\begin{array}{ll}
204q^{-3}&\text{ for } n=8, \\[0.3cm]
252q^{-4} &\text{ for } n=10, \\[0.3cm]
490q^{-8} &\text{ for } n=12.
\end{array}\right.
\end{equation*}
Combining this with (\ref{e03}) and (\ref{go1-}), we obtain \eqref{eq:PO-} and thus complete the proof of Proposition \ref{PO-}.

\medskip

\noindent \textbf{Acknowledgements}~~S. Zheng was supported by the Melbourne Research Scholarship.

\end{document}